\documentclass[10pt]{article}
\usepackage[intlimits]{amsmath}
\usepackage{amsthm}
\usepackage{amsfonts}
\usepackage{amssymb}

\font\notefont=cmsl8

\theoremstyle{plain}
\newtheorem{thm}{THEOREM}[section]
\newtheorem{lm}[thm]{LEMMA}

\theoremstyle{definition}

\theoremstyle{remark}                                %
\newcommand{\R}{{\mathord{\mathbb R}}}
\newcommand{\RR}{{\mathord{\mathbb R}}^n}
\newcommand{\SN}{\mathord{\mathbb S}}

\newcommand{\HN}{\mathord{\RR_+}}
\newcommand{\ud}{\mspace{3mu}\mathrm{d}}
\newcommand{\supp}{{\mathop{\rm supp\ }}}

\newcommand{\ez}{\epsilon\rightarrow0}

\newcommand{\ft}{\widetilde{f}}

\newcommand{\IH}{I_{\alpha}^{\HN}}
\newcommand{\JH}{J_{\alpha}^{\HN}}
\newcommand{\RQ}{\Phi_{\alpha}}
\newcommand{\dvx}{\frac{\ud x}{x_n^{(1-\alpha)/2}}}
\newcommand{\dvy}{\frac{\ud y}{y_n^{(1-\alpha)/2}}}
\newcommand{\inti}{\int_0^{\infty}}

\begin{document}
\allowdisplaybreaks

\title{\bf{A Fractional Hardy-Sobolev-Maz'ya Inequality on the Upper Halfspace}}
\author{Craig A. Sloane\\
\small{Georgia Institute of Technology, Atlanta, Georgia 30332-0160}\\
\small{csloane@math.gatech.edu}\\}
\date{\today}
\maketitle

\begin{abstract}
We prove several Sobolev inequalities, which are then used to establish a fractional Hardy-Sobolev-Maz'ya inequality on the upper halfspace.
\end{abstract}

\section{Introduction}
The present work answers a question by Frank and Seiringer \cite{FS2} concerning fractional Hardy-Sobolev-Maz'ya inequalities for the upper halfspace in the case $p=2$. Let
\[\HN = \left\{x = \left(x',x_n\right) \in\RR : x' \in \R^{n-1}, x_n > 0\right\}\]
be the upper halfspace, and let $\Omega$ be a domain in $\RR$ with nonempty boundary. Then, there exists a fractional Hardy inequality on $\HN$ which states that there exists $D_{n,p,\alpha} > 0$ so that for all $f \in C_c(\HN)$,
\begin{equation}
\label{hardy}
\int_{\HN\times\HN}{\frac{|f(x)-f(y)|^p}{|x-y|^{n+\alpha}}\ud x \ud y} \geq D_{n,p,\alpha}\int_{\HN}{\frac{|f(x)|^p}{x_n^{\alpha}} \ud x},
\end{equation}
where $1 \leq p < \infty$, $0 < \alpha < p$, and $\alpha \neq 1$. See, \emph{e.g.}, \cite{D}. Bogdan and Dyda found, in \cite{BD}, the sharp	constant $D_{n,2,\alpha}$ for the $p=2$ case. Later, the sharp constant, $D_{n,p,\alpha}$, for general $p$ was found in \cite{FS2}. Therein, the authors posed the question whether there existed a lower bound to the remainder for the inequality in \eqref{hardy} that is a positive multiple of the $L^{p^*}$-norm of $f$, where $p^* = np/(n-\alpha)$ is the critical Sobolev exponent. Such an inequality would be a fractional analogue to the Hardy-Sobolev-Maz'ya inequality on the halfspace. Maz'ya \cite{M} was the first to show the general result
\[
\int_{\HN}{|\nabla f(x)|^2 \ud x} - \frac{1}{4}\int_{\HN}{\frac{|f(x)|^2}{x_n^2} \ud x}
			\geq C_n\left(\int_{\HN}{|f(x)|^{\frac{2n}{n-2}} \ud x}\right)^{\frac{n-2}{n}}.
\]
More recently the existence of minimizers for dimensions greater than or equal to 4 \cite{TT} and the sharp constant for dimension 3 \cite{BFL} has been established. Further improvements in the general case have been shown in \cite{FTT}. 

The following Theorem for the sharp fractional hardy inequality with remainder was proven by Frank and Seiringer in \cite{FS2}.

\begin{thm}
Let $n \ge 1$, $2 \le p \le \infty$, and $0 < \alpha < p$ with $\alpha \ne 1$. Then for all $f \in C_c^{\infty}(\HN)$,
\begin{align}
\label{remainder}
\nonumber
\int_{\Omega\times\Omega}{\frac{|f(x)-f(y)|^p}{|x-y|^{n+\alpha}}\ud x \ud y}
&-			D_{n,p,\alpha} \int_{\HN}{\frac{|f(x)|^p}{x_n^{\alpha}} \ud x}\\
&\ge		c_p \int_{\Omega\times\Omega}{\frac{|x_n^{(1-\alpha)/p}f(x)-y_n^{(1-\alpha)/p}f(y)|^p}{|x-y|^{n+\alpha}} \dvx \dvy}
\end{align}
where $0 < c_p \le 1$ is given by 
\[
c_p := \min_{0 < \tau < 1/2}\left((1-\tau)^p - \tau^p + p\tau^{p-1}\right).
\]
If $p=2$ then this is an equality with $c_p = 1$.
\end{thm}
For notational convenience, we write
\[
I_{\alpha,p}^{\Omega}(f)
=			\int_{\Omega\times\Omega}{\frac{|f(x)-f(y)|^p}{|x-y|^{n+\alpha}}\ud x \ud y},
\]
and
\[
J_{\alpha,p}^{\Omega}(f)
=			\int_{\Omega\times\Omega}{\frac{|x_n^{(1-\alpha)/p}f(x)-y_n^{(1-\alpha)/p}f(y)|^p}{|x-y|^{n+\alpha}} \dvx \dvy}.
\]
Since we are primarily concerned with the case $p=2$, we further denote $I_{\alpha}^{\Omega} = I_{\alpha,2}^{\Omega}$ and $J_{\alpha}^{\Omega} = J_{\alpha,2}^{\Omega}$.
Thus, for $p=2$, we can rewrite \eqref{remainder} as
\begin{equation}
\IH(f) - D_{n,2,\alpha}\int_{\HN}{|f(x)|^2 x_n^{-\alpha} \ud x}
=				\JH(f),
\end{equation}

The main result of this paper is the following fractional Hardy-Sobolev-Maz'ya inequality for $p=2$.

\begin{thm}
\label{hsmthm}
Let $n \geq 2$, $1<\alpha<2$. There exists $a_{n,\alpha} > 0$ so that
\begin{equation}
\label{hsm}
\int_{\HN\times\HN}{\frac{|f(x)-f(y)|^2}{|x-y|^{n+\alpha}}\ud x \ud y} - D_{n,2,\alpha}\int_{\HN}{\frac{|f(x)|^2}{x_n^{\alpha}} \ud x}
			\geq a_{n,\alpha}\left(\int_{\HN}{|f(x)|^{2^*} \ud x}\right)^{2/{2^*}},
\end{equation}
for all $f \in C_c^{\infty}(\HN)$.
\end{thm}
Alternatively, we write \eqref{hsm} as $\JH(f) \ge a_{n,\alpha}\|f\|_{2^*}^2$, where $\|\cdot\|_p$ refers to the $L^p$-norm with usual Lebesgue measure.

\section{Sobolev Inequalities}

Herein, we establish two Sobolev-type inequalities that we'll need in the proof of Theorem \ref{hsmthm}. We prove each for the more general $p$-case. The first inequality we prove is for $I_{\alpha,p}^{\Omega}$ with respect to convex sets $\Omega$.

\begin{thm}
\label{sobolev}
Let $p \geq 2$, $1<\alpha<\min\{n,p\}$, and let $\Omega \subseteq \RR$ be convex.  Then, for all $f \in C_c^{\infty}(\Omega)$, there exists $c_{n,p,\alpha}>0$ so that
\[
I_{\alpha,p}^{\Omega}(f) \geq c_{n,p,\alpha}\|f\|_{p^*}^p.
\]
\end{thm}

\begin{proof}
In \cite{LS}, it is shown
\[
I_{\alpha,p}^{\Omega}(f)
\geq			D_{n,p,\alpha}\int_{\Omega}{|f(x)|^p d_{\Omega}(x)^{-\alpha} \ud x},
\]
for all $f \in C_c(\Omega)^{\infty}$, where $d_{\Omega}(x) = \mathrm{dist}(x,\partial\Omega)$. Further, there exists $S_{n,p, \alpha}$, see, \emph{e.g.}, \cite{AF}, Theorems 7.34, 7.47, so that 
\[
I_{\alpha,p}^{\RR}(f)
\geq			S_{n,p,\alpha}\|f\|_{p^*}^p,
\]
for all $f \in C_c^{\infty}(\RR)$. Thus, if $B(x,r)$ is the ball of radius $r$ centered at $x$, then since
\[\int_{\Omega^c}{|x-y|^{-n-\alpha} \ud y}\;
			\leq	\int_{B(x,d_{\Omega}(x))^c}{|x-y|^{-n-\alpha} \ud y}
			=			\tfrac{1}{\alpha}|\SN^{n-1}|d_{\Omega}(x)^{-\alpha},\]
we have
\begin{align*}
S_{n,p,\alpha}\|f\|_{p^*}^p
&\leq		I_{\alpha,p}^{\RR}(f)\\
&= 			I_{\alpha,p}^{\Omega}(f) + 2\int_{\Omega}{\ud x |f(x)|^p}\int_{\Omega^c}{\ud y |x-y|^{-n-\alpha}}\\
&\leq		I_{\alpha,p}^{\Omega}(f) + \tfrac{2}{\alpha}|\SN^{n-1}|\int_{\Omega}{|f(x)|^p d_{\Omega}(x)^{-\alpha} \ud x}\\
&\leq		\left(1 + \frac{2|\SN^{n-1}|}{\alpha D_{n,p, \alpha}}\right) I_{\alpha,p}^{\Omega}(f),
\end{align*}
where $|\SN^{n-1}| = \frac{2\pi^{n/2}}{\Gamma(n/2)}$ is the surface area of the sphere of radius 1 in $\RR$.
\end{proof}

The next inequality is a weighted inequality for the term $J_{\alpha,p}^{\HN}(f)$. We leave the proof to the appendix.

\begin{thm}
\label{weighted}
Let $p \geq 2$, $1<\alpha<\min\{n,p\}$. Then, there exists $d_{n,p,\alpha}>0$ so that
\[J_{\alpha,p}^{\HN}(f) \geq d_{n,p,\alpha}\left(\int_{\HN}{|f(x)|^q x_n^{-n+nq/{p^*}} \ud x}\right)^{p/q},\]
where $q=p\left(\frac{n+\frac{\alpha-1}{2}}{n-1}\right)$.
\end{thm}

Since the remainder of the paper deals with the case $p=2$, we write $c_{n,\alpha} = c_{n,2,\alpha}$ and $d_{n,\alpha} = d_{n,2,\alpha}$.

In the next section, we'll show we can minimize over a certain class of functions that are decreasing, albeit not symmetrically. The crux of the proof of Theorem \ref{hsmthm} in the following section is to decompose this function by truncation and use these two Sobolev inequalities to appropriately bound the $L^{2^*}$-norms of the resulting ``upper'' and ``lower'' functions.

\section{Class of Minimizing Functions}
In this section, we determine the properties of those functions that minimize our Rayleigh quotient
\[
\RQ(f)
:=		\frac{\IH(f) - D_{n,2,\alpha}\int_{\HN}{|f(x)|^2 x_n^{-\alpha} \ud x}}{\|f\|_{2^*}^2}
=			\frac{\JH(f)}{\|f\|_{2^*}^2}.
\]
To minimize this functional, we'd like to do a rearrangement, but we have the restriction that the function must have support in the upper halfspace. Still, we can rearrange the function along hyperplanes parallel to the boundary of $\HN$.

In addition, we consider the conformal transformation $T: B(\textbf{0},1) \rightarrow \HN$, where $\textbf{0}$ is the origin in $\RR$. If we write $\eta(\omega) = 2/\left(\left|\omega'\right|^2 + (\omega_n + 1)^2\right)$, where $\omega=(\omega',\omega_n)$, $\omega'\in\R^{n-1}$, $\omega_n\in\R$, then
\[T\omega		= \left(\frac{2\omega',1-|\omega|^2}{\left|\omega'\right|^2 + (\omega_n + 1)^2}\right)
						=	\eta(\omega)\left(\omega',\frac{1-|\omega|^2}{2}\right).\]
Note that $T$ is an involution, and its Jacobian is $\eta(\omega)^n$. See, \emph{e.g.}, Appendix, \cite{CL2}. We define
\[\ft(\omega) := \eta(\omega)^{n/{2^*}}f\left(T\omega\right),\]
and since $\eta(Tx)=1/\eta(x)$, then
\[f(x) = \eta(x)^{n/{2^*}}\ft(Tx).\]
Thus, if $\supp \ft \subseteq B(\textbf{0},1)$, then $\supp f \subseteq \HN$. Indeed, for any $0 < R < 1$, if $\supp \ft \subseteq B(\textbf{0},R)$, then $\supp f \subseteq	B^R$, where we define
\begin{equation}
\label{BR}
B^R :=		\left\{(x',x_n)\in\RR : |x'|^2 + \left(x_n - \frac{1+R^2}{1-R^2}\right)^2 \leq \left(\frac{2R}{1-R^2}\right)^2\right\},
\end{equation}
using $|Tx|^2 = \dfrac{|x'|^2+(x_n - 1)^2}{|x'|^2+(x_n + 1)^2}$.

We use these results throughout the remainder of this paper. These provide a new ``ball'' picture in which to consider our inequality and minimization problem. Among other things, we can also perform a rotation of $\ft$ on the ball. It turns out that repeated application of this rotation, along with the rerrangement mentioned above, result in a limiting function that is radial in the ball picture and whose Rayleigh quotient is always smaller than that of the original.

This first result follows from Carlen and Loss \cite{CL2}.
\begin{thm}
\label{rearrange}
Let $n \geq 2$, $f \in L^{2^*}(\HN)$. Then, there exists $F \in L^{2^*}(\HN)$ such that 
\begin{enumerate}
	\item	$\left\|f\right\|_{2^*} = \left\|F\right\|_{2^*}$,
	\item	$F$ is nonnegative, symmetric decreasing in hyperplanes parallel to the boundary of $\HN$,
	\item	$\widetilde{F}$ is rotationally symmetric, and
	\item	$\RQ(f) \geq \RQ(F)$.
\end{enumerate}
\end{thm}

\begin{proof}
Since $|f(x)-f(y)| \geq \Bigl||f|(x)-|f|(y)\Bigr|$ implies $\IH(f) \geq \IH(|f|)$, we can assume $f$ is nonnegative. Then, the first three items of this theorem are a direct result of \cite{CL2}, Theorem 2.4. Indeed, let $Uf$ be the transformation of $f$ obtained by a certain fixed rotation of $\ft$, as described in \cite{CL2}. In particular, using the rotation
$\textbf{R}:(x_1,\ldots,x_{n-1},x_n) \mapsto (x_1,\ldots,x_n, -x_{n-1})$, $x_i \in \R, i = 1,\ldots,n$,
then $U$ maps
\[f(x) \mapsto \ft(x) \mapsto \ft(\textbf{R}x) \mapsto \eta(x)^{n/{2^*}}\ft(\textbf{R}Tx).\]
Note how the last transformation mimics the map $\ft \mapsto f$. Further, let $V$ be the symmetric decreasing rearrangement in hyperplanes parallel to the boundary of $\HN$. Define $F_k := (VU)^{k}f$, then, by Theorem 2.4 in \cite{CL2}, there exists $F \in L^{2^*}(\HN)$ that is nonnegative and symmetric decreasing in hyperplanes parallel to the boundary of $\HN$ and such that $\left\|f\right\|_{2^*} = \left\|F\right\|_{2^*}$, $\widetilde{F}$ is radial on the unit ball, and $\lim_{k\rightarrow\infty} F_k = F$ in $L^{2^*}(\HN)$. By passing to a subsequence, we can assume, without loss of generality, $F_k \rightarrow F$ almost everywhere.

As was calculated in \cite{BD}, there exists a constant $c > 0$ so that we can write the remainder term as
\begin{equation}
\label{decomp}
\JH(f) = I_{\alpha}^{\RR}(f) - c\int_{\HN}{\frac{|f(x)|^2}{x_n^{\alpha}}\ud x}.
\end{equation}	
We claim $\RQ(F_k)$ is decreasing. As $F_0 = f$, it is enough to show that $\RQ(f) \geq \RQ(VUf)$. 

By a modification of Theorem 7.17 of \cite{LL}, $I_{\alpha}^{\RR}(f)$ decreases under the rearrangement $V$. From Lemma \ref{lemma} in the Appendix, $I_{\alpha}^{\RR}(f) = I_{\alpha}^{\RR}(\ft)$, the latter of which is invariant under rotations. Hence, $I_{\alpha}^{\RR}(f)$ is invariant under $U$. Next, as the rearrangement under $V$ is only along hyperplanes parallel to the boundary of $\HN$ (\emph{i.e.}, where $x_n$ is fixed), the integral $\int_{\HN}{f^2(x) x_n^{-\alpha}\ud x}$ must be invariant under $V$, and since 
\[
\int_{\HN}{f^2(x) x_n^{-\alpha}\ud x}
=			\int_{B(\textbf{0},1)}{\left(\frac{2}{1-|\omega|^2}\right)^{\alpha} \ft^2(\omega)\ud \omega},
\]
it is invariant under $U$ as well. Finally, the $L^{2^*}$-norm is clearly invariant under $U$ and $V$.

Therefore, applying Fatou's lemma,
\[
\RQ(f)
=			\frac{\JH(f)}{\|f\|_{2^*}^2}
=			\frac{\JH(F_0)}{\|F\|_{2^*}^2}
\geq	\lim_{k \rightarrow \infty}\frac{\JH(F_k)}{\|F\|_{2^*}^2}
\geq	\frac{\JH(F)}{\|F\|_{2^*}^2}
=			\RQ(F).\qedhere
\]
\end{proof}

As a result, we can explicitly write the limiting function in Theorem \ref{rearrange} as the product of two radial functions in the ball picture, a specific, known symmetrically increasing function and a symmetrically decreasing function.

\begin{thm}
\label{representation}
Let $n \ge 2$, $F \in L^{2^*}(\HN)$, where $F$ is nonnegative, symmetric decreasing in hyperplanes parallel to the boundary of $\HN$, and $\widetilde{F}$ is rotationally symmetric. Then, there exists a decreasing function $h: [0,1] \rightarrow [0,\infty]$, where $h(1) = 0$, so that
\[
\widetilde{F}(\omega)
=			\left(\frac{2}{1-|\omega|^2}\right)^{n/{2^*}} h(|\omega|).
\]
\end{thm}

\begin{proof}
Let $x \in \HN$ such that $x_n = 1$, and, recalling that $T$ is an involution, let $\omega = Tx$. Thus, if we restrict $T$ to the hyperplane
\[
H =	\{(x',x_n) \in \HN : x'\in\R^{n-1}, x_n = 1\},
\]
then its image, or stereographic projection, is the sphere
\[
S = \left\{\omega=(\omega',\omega_n):\omega'\in\R^{n-1}, \omega_n\in\R, |\omega'|^2 + \left(\omega_n + \tfrac{1}{2}\right)^2 = \tfrac{1}{4}\right\}
\]
whose north and south poles pass through the origin and the point $(0,\ldots,0,-1)$, respectively. Thus, for all $\omega \in S$, we have that $\omega_n = -|\omega|^2$, and $\eta(\omega) =  2/(1-|\omega|^2)$. Further,
\[
\frac{2}{1-|\omega|^2}
=			\eta(\omega)
=			\eta(Tx)
=			\frac{1}{\eta(x)}
=			\frac{|x'|^2+4}{2},
\]
as $x \in H$. Then, $|x'|^2 = 4|\omega|^2/(1-|\omega|^2)$, and, since $F$ is radial on $H$, then, for all $\omega \in S$,
\[
F(T\omega)
=			F(x',1)
=			F\left(|x'|,1\right)
=			F\left(\frac{2|\omega|}{\sqrt{1-|\omega|^2}},1\right).
\]
Further, as $\widetilde{F}(\omega) = \eta(\omega)^{n/{2^*}} F(T\omega)$, then
\[
\widetilde{F}(\omega) = \left(\frac{2}{1-|\omega|^2}\right)^{n/{2^*}}\mspace{-18.0mu}F\left(\frac{2|\omega|}{\sqrt{1-|\omega|^2}},1\right)
											=	\left(\frac{2}{1-|\omega|^2}\right)^{n/{2^*}}\mspace{-18.0mu}h(|\omega|),
\]
where $h(r) = F\left(\frac{2r}{\sqrt{1-r^2}},1\right)$ is a decreasing function and $h(1) = 0$, as $F$ is radially symmetrically decreasing on $H$, and $F \in L^{2^*}(\HN)$.

Note that for each particular radius in the unit ball, the corresponding sphere intersects $S$. This radius then corresponds to a particular radius in $H$ as given by $|x'|^2 = 4|\omega|^2/(1-|\omega|^2)$. But, on the ball, $\widetilde{F}$ is a radial function, so if $\omega$ is any point in the unit ball, there exists some rotation $R_{\omega}$ and $\omega_S \in S$ so that $\omega = R_{\omega}\omega_S$. Therefore,
\[\widetilde{F}(\omega) = \widetilde{F}(R_{\omega}\omega_S) 
												\left(\frac{2}{1-|R_{\omega}\omega_S|^2}\right)^{n/{2^*}}\mspace{-18.0mu}h(|R_{\omega}\omega_S|)
												= \left(\frac{2}{1-|\omega|^2}\right)^{n/{2^*}}\mspace{-18.0mu}h(|\omega|),\]
for all $\omega \in B(\textbf{0},1)$.
\end{proof}

\section{Proof of Main Result}
\noindent\textit{Proof of Theorem \ref{hsmthm}:} From Theorems \ref{rearrange} and \ref{representation}, we can assume $\ft(\omega) = \left(\frac{2}{1-|\omega|^2}\right)^{n/{2^*}}h(|\omega|)$, where $h(r)$ is a decreasing function on $[0,1]$ and $h(1)=0$. Then,
\[
f(x)
=			\eta(x)^{n/{2^*}}\ft(Tx)
=			x_n^{-n/{2^*}}h(|Tx|),
\]
and $I_{\alpha}^{\RR}(f), \|f\|_{2^*}^2 < \infty$. However, we note that $f$ is no longer necessarily in $C_c^{\infty}(\HN)$.

We decompose $h = h_1 + h_0$ by truncation, by fixing $R \in (0,1)$ so $h_0(r) = \min\{h(r),h(R)\}$. Then, $f = f_1 + f_0$, with the definitions $f_1, f_0$ following from the above. We claim there exists $c,d > 0$, each dependent on $R,n$ and $\alpha$, such that
\begin{equation}
\label{claim1}
\JH(f) \geq c \left\|f_1\right\|_{2^*}^2,
\end{equation}
and
\begin{equation}
\label{claim2}
\JH(f) \geq d \left\|f_0\right\|_{2^*}^2.
\end{equation}
Then, using the triangle and arithmetic-geometric mean inequalities, for all $0<\lambda<1$, we obtain
\[\JH(f)	\geq \lambda c\left\|f_1\right\|_{2^*}^2 + (1-\lambda) d\left\|f_0\right\|_{2^*}^2
					\geq \tfrac{1}{2}\min\{\lambda c, (1-\lambda)d\}\left\|f\right\|_{2^*}^2.\]
Clearly, by fixing $\lambda, R$ not equal to  zero or one, the constant is greater than zero.  So, by taking the supremum over $\lambda$ and $R$, the result follows.

First, we prove \eqref{claim1}. Note that $\supp h_1 \subseteq [0,R]$, so $\supp f_1 \subseteq B^R$, where $B^R$ is as in \eqref{BR} above. Thus, for all $x,y \in B^R$,
\begin{align*}
\left|x_n^{(1-\alpha)/2}f_1(x) - y_n^{(1-\alpha)/2}f_1(y)\right|^2
&=		\left|x_n^{\frac{1-n}{2}}\left(h(|Tx|)-h(R)\right) - y_n^{\frac{1-n}{2}}\left(h(|Ty|)-h(R)\right)\right|^2\\
&=		\left|\left(x_n^{\frac{1-n}{2}}h(|Tx|) - y_n^{\frac{1-n}{2}}h(|Ty|)\right) + h(R)\left(y_n^{\frac{1-n}{2}} - x_n^{\frac{1-n}{2}}\right)\right|^2\\
&\leq	2\left|x_n^{(1-\alpha)/2}f(x) - y_n^{(1-\alpha)/2}f(y)\right|^2 + 2h^2(R)\left|y_n^{\frac{1-n}{2}} - x_n^{\frac{1-n}{2}}\right|^2.
\end{align*}
It is easy to see that for any $0<R<1$,
\[
A_1		= \int_{B^R\times B^R}{\frac{\left|y_n^{(1-n)/2}-x_n^{(1-n)/2}\right|^2}{|x-y|^{n+\alpha}}
				x_n^{\frac{\alpha-1}{2}}y_n^{\frac{\alpha-1}{2}} \ud x \ud y}
			<	\infty,
\]
where $A_1$ is dependent on $R, n$ and $\alpha$. Thus, $J_{\alpha}^{B^R}(f) + A_1 h^2(R) \geq \frac{1}{2}J_{\alpha}^{B^R}(f_1)$. We claim we can apply Theorem \ref{sobolev} to $x_n^{\frac{1-\alpha}{2}}f_1(x)$. Hence, if $x \in B^R$, then $\frac{1-R}{1+R} \leq x_n \leq \frac{1+R}{1-R}$, and
\begin{align*}
J_{\alpha}^{B^R}(f_1)
&\geq		\left(\frac{1-R}{1+R}\right)^{\alpha - 1} I_{\alpha}^{B^R}\left(x_n^{\frac{1-\alpha}{2}}f_1(x)\right)\\
&\geq		c_{n,\alpha}\left(\frac{1-R}{1+R}\right)^{\alpha - 1}
				\left(\int_{B^R}{x_n^{\left(\frac{1-\alpha}{2}\right){2^*}}\left|f_1(x)\right|^{2^*} \ud x}\right)^{2/{2^*}}\\
&\geq		c_{n,\alpha}\left(\frac{1-R}{1+R}\right)^{2\alpha - 2} \left\|f_1\right\|_{2^*}^2.
\end{align*}
Using Theorem \ref{weighted},
\[\JH(f)
\geq		d_{n,\alpha}\left(\int_{\HN}{\bigl|h(|Tx|)\bigr|^{q} x_n^{-n} \ud x}\right)^{2/{q}}
\geq		 d_{n,\alpha} h^2(R) \left(\int_{B^R}{x_n^{-n} \ud x}\right)^{2/{q}}
=				 A_2 h^2(R),\]
where $A_2$ is also dependent on $R,n$ and $\alpha$. Therefore,
\[\left(1+\frac{A_1}{A_2}\right)\JH(f)
		\geq		J_{\alpha}^{B^R}(f) + A_1 h^2(R)
		\geq		\tfrac{1}{2} c_{n,\alpha}\left(\frac{1-R}{1+R}\right)^{2\alpha - 2} \left\|f_1\right\|_{2^*}^2,\]
which proves \eqref{claim1}.
 
In establishing \eqref{claim2}, we use the inequality
$\frac{1}{2(n-1)} \leq \frac{(1-S^2)^{n-1}}{S^n}\int_0^S{\frac{r^{n-1}}{(1-r^2)^n} \ud r} \leq \frac{1}{n-1}$, $0<S<1$. Note that $h_0$ is constant on $[0,R]$, while it is decreasing to zero on $[R,1]$. The following establishes how fast $h_0$ vanishes at 1. From Theorem \ref{weighted},
\begin{align*}
\JH(f)
&\geq		d_{n,\alpha}\left(\int_{\HN}{\bigl|h_0(|Tx|)\bigr|^{q} x_n^{-n} \ud x}\right)^{2/{q}}\\
&=			d_{n,\alpha} \left(2^n |\SN^{n-1}|\int_0^1{\frac{r^{n-1}}{(1-r^2)^n} h_0(r)^{q} \ud r}\right)^{2/{q}}\\
&\geq		d_{n,\alpha} h_0(S)^2 \left(\frac{2^{n-1}}{n-1}|\SN^{n-1}|\right)^{2/{q}}\left(\frac{S^n}{(1-S^2)^{n-1}}\right)^{2/{q}},
\end{align*}
where $0<S<1$. Thus,
\[
h_0(r)^{2^*}	\leq	d_{n,\alpha}^{-{2^*}/2}\left(\frac{2^{n-1}}{n-1}|\SN^{n-1}|\right)^{-{2^*}/q}
										\left(\frac{(1-r^2)^{n-1}}{r^n}\right)^{{2^*}/q}\JH(f)^{{2^*}/2}.
\]
Then, we calculate
\begin{align*}
\left\|f_0\right\|_{2^*}^{2^*}
&=			2^n |\SN^{n-1}|\int_0^1{\frac{r^{n-1}}{(1-r^2)^n} h_0(r)^{2^*} \ud r}\\
&=			2^n |\SN^{n-1}| \left(h_0(R)^{2^*} \int_0^R{\frac{r^{n-1}}{(1-r^2)^n} \ud r} + \int_R^1{\frac{r^{n-1}}{(1-r^2)^n} h_0(r)^{2^*} \ud r}\right)\\
&\leq		2^n |\SN^{n-1}|d_{n,\alpha}^{-{2^*}/2}\left(\frac{2^{n-1}}{n-1}|\SN^{n-1}|\right)^{-{2^*}/q}
						\left(\left(\frac{(1-R^2)^{n-1}}{R^n}\right)^{{2^*}/q}\frac{1}{n-1}\frac{R^n}{(1-R^2)^{n-1}} \right.\\
&\qquad	\left. + \int_R^1{\frac{r^{n-1}}{(1-r^2)^n} \left(\frac{(1-r^2)^{n-1}}{r^n}\right)^{{2^*}/q}\ud r}\right)\JH(f)^{{2^*}/2}.
\end{align*}
As ${2^*} > q$, the claim follows.

Finally, we show that we can approximate $x_n^{\frac{1-\alpha}{2}}f_1(x)$ by functions in $C_c^{\infty}(B^R)$. Define
\[g_c(x) = \max\{x_n^{\frac{1-\alpha}{2}}f_1(x) - c, 0\}\]
almost everywhere. Then, by monotone convergence,
\[I_{\alpha}^{B^R}(g_c) \rightarrow I_{\alpha}^{B^R}\left(x_n^{\frac{1-\alpha}{2}}f_1(x)\right)\]
and
\[\|g_c\|_{2^*}^2 \rightarrow \left\|x_n^{\frac{1-\alpha}{2}}f_1(x)\right\|_{2^*}^2\]
as $c \rightarrow 0$. Now, $\supp g_c \subsetneq B^R$ and $g_c \in L^2(B^R)$, so $I_{\alpha}^{\RR}(g_c) < \infty$ from \eqref{decomp}. Denote
\[\|\cdot\|_{W^{\alpha/2,2}(\RR)} = \sqrt{\|\cdot\|_2^2+I_{\alpha}^{\RR}(\cdot)},\]
and let $W_0^{\alpha/2,2}(\RR)$ be the completion of $C_c^{\infty}(\RR)$ with respect to $\|\cdot\|_{W^{\alpha/2,2}(\RR)}$. Then it is known that
\[
W_0^{\alpha/2,2}(\RR)
=			W^{\alpha/2,2}(\RR) = \left\{ u \in L^2(\RR) : \|u\|_{W^{\alpha/2,2}(\RR)} < \infty \right\},
\]
see, \emph{e.g.} \cite{AF}, \cite{BBC}. Since $\supp g_c \subsetneq B^R$, there exists a sequence $\{g_c^j\} \subset C_c^{\infty}(B^R)$ so that $\|g_c - g_c^j\|_{W^{\alpha/2,2}(\RR)} \rightarrow 0$ as $j \rightarrow \infty$. Hence, $I_{\alpha}^{B^R}(g_c^j) \rightarrow I_{\alpha}^{B^R}(g_c)$ and $\|g_c^j\|_2^2 \rightarrow \|g_c\|_2^2$ as $j \rightarrow \infty$.
\qed

\section{Conclusion}
Consider the general Hardy-Sobolev-Maz'ya inequality
\[
\int_{\HN\times\HN}{\frac{|f(x)-f(y)|^p}{|x-y|^{n+\alpha}}\ud x \ud y} - D_{n,p,\alpha}\int_{\HN}{\frac{|f(x)|^p}{x_n^{\alpha}} \ud x}
\geq			a_{n,p,\alpha}\left(\int_{\HN}{|f(x)|^{p^*} \ud x}\right)^{p/{p^*}},
\]
where $p \ge 2$, $1<\alpha<\min\{n,p\}$. It is still unknown whether $a_{n,p,\alpha} > 0$ for $p>2$. Still, other than Theorem \ref{rearrange}, the elements of this paper have either already been proven for general $p$ or extend quite easily. Indeed, Theorem \ref{rearrange} cannot be extended because Lemma \ref{lemma} is not true for $p \ne 2$. However, for functions that are symmetrically decreasing about a point in the upper halfspace or that can be rerranged about a point in the upper halfspace, while still maintaining support there, and where the Hardy term increases due to the rearrangement, then a fixed $a_{n,p,\alpha} > 0$ can be found.

\smallskip
\noindent \textbf{Acknowledgement.} I am very thankful to Michael Loss for countless valuable discussions and especially for collaboration on Theorem \ref{weighted}. This work was partially supported by NSF Grant DMS 0901304.

\section{Appendix}
\noindent\textit{Proof of Theorem \ref{weighted}:}
This proof uses the idea from Theorem 4.3, \cite{LL} to write the integral in terms of its layer cake representation . We can assume that $f\geq0$, since
\[
\JH(f)
=			\IH(f) - 2\kappa_{n,\alpha}\int_{\HN}{\frac{|f(x)|^2}{x_n^{\alpha}}\ud x}\\
\geq	\IH(|f|) - 2\kappa_{n,\alpha}\int_{\HN}{\frac{|f|(x)^2}{x_n^{\alpha}}\ud x}\\
=			\JH(|f|).
\]
We need a few preliminary results. Let $1_{\Omega}$ be the indicator function on the set $\Omega$, then, for any $s\in\R$,
\[
\int_0^{\infty}{st^{-s-1} 1_{\left\{|x|<t\right\}}\ud t}
=			\int_{|x|}^{\infty}{st^{-s-1}\ud t} = |x|^{-s}.
\]
These next results follow from the Appendix in \cite{CL2}. Let $t \ge 0$, so $t^p = {p(p-1)\int_0^{\infty}(t-a)_+ a^{p-2} \ud a}$. Then, letting $a \ge 0$,
\begin{align*}
|g(x)-g(y)|^p
&=			p(p-1)\int_0^{\infty}{\left(|g(x)-g(y)|-a\right)_+ a^{p-2} \ud a}\\
&=			p(p-1)\int_0^{\infty}{\bigl[\left(g(x)-g(y)-a\right)_+ + \left(g(y)-g(x)-a\right)_+\bigr] a^{p-2} \ud a}\\
&=			p(p-1)\int_0^{\infty}{\ud a \;a^{p-2}}\int_0^{\infty}{\ud b \left(1_{\{g(x)-a>b\}}1_{\{g(y)<b\}}+1_{\{g(y)-a>b\}}1_{\{g(x)<b\}}\right)},
\end{align*}
where $1_A$ is the indicator function on $A$. Let $g(x) = x_n^{\frac{1-\alpha}{p}}f(x)$. Then, where $d_{n,p,\alpha}$ is a generic constant, and using the results above,
\begin{align*}
J_{\alpha,p}^{\HN}(f)
&=			\int_{\HN\times\HN}{\frac{|g(x)-g(y)|^p}{|x-y|^{n+\alpha}} \dvx \dvy}\\
&=			d_{n,p,\alpha} \int_{\HN\times\HN}{\dvx \dvy} \inti{\frac{\ud c}{c} \;c^{-n-\alpha} 1_{\{|x-y|<c\}}} \inti{\ud a \;a^{p-2}}
				\inti{\ud b} \Bigl(1_{\{g(x)>a+b\}}1_{\{g(y)<b\}}\Bigr.\\
&\mspace{600.0mu}	\Bigl. +1_{\{g(y)>a+b\}}1_{\{g(x)<b\}}\Bigr)\\
&=			d_{n,p,\alpha} \mspace{-9.0mu}\int_{\HN\times\HN}{\dvx \dvy} \inti{\frac{\ud c}{c} \;c^{-n-\alpha}} \inti{\ud a \;a^{p-2}}
				\inti{\ud b  \;1_{\{|x-y|<c\}}\left(1-1_{\{g(y) \ge b\}}\right)1_{\{g(x)>a+b\}}}.
\end{align*}
We write
\[\lambda(a) = \int_{\HN}{1_{\{g(x)>a\}} \dvx},\]
and
\[u(a,c) = \int_{\HN\times\HN}{1_{\{g(x)>a\}} 1_{\{|x-y|<c\}} \dvx \dvy},\]
so $u(a,c) \ge u(b,c)$, $\lambda(a) \ge \lambda(b)$ if $b \ge a$. Further, since $\alpha > 1$, we obtain
\begin{align*}
u(a,c)
&=			\int_{\HN}{\dvx 1_{\left\{g(x)>a\right\}}} \int_{\HN}{\dvy 1_{\left\{|x-y|<c\right\}}}\\
&\geq		\int_{\HN}{\dvx 1_{\left\{g(x)>a\right\}}} \int_{\HN}{\dvy 1_{\left\{|y|<c\right\}}}\\
&=			Dc^{n+\frac{\alpha-1}{2}} \lambda(a),
\end{align*}
where $D=\int_{\HN}{1_{\left\{|y|<1\right\}} \dvy}$. Using Fubini,
\pagebreak[0]
\begin{align*}
J_{\alpha,p}^{\HN}
&=			d_{n,p,\alpha} \inti{\ud a \;a^{p-2}} \inti{\ud b} \inti{\frac{\ud c}{c} \;c^{-n-\alpha}}
				\int_{\HN\times\HN}{\dvx \dvy} \Bigl(1_{\{|x-y|<c\}} 1_{\{g(x)>a+b\}}\Bigr.\\
&\mspace{540.0mu} \Bigl.- 1_{\{|x-y|<c\}} 1_{\{g(x)>a+b\}} 1_{\{g(y) \ge b\}}\Bigr)\\
&\ge		d_{n,p,\alpha} \inti{\ud a \;a^{p-2}} \inti{\ud b} \inti{\frac{\ud c}{c} \;c^{-n-\alpha}}
				\Bigl(u(a+b,c) - \min\{u(a+b,c),u(b,c),\lambda(a+b)\lambda(b)\}\Bigr)\\
&\ge			d_{n,p,\alpha} \inti{\ud a \;a^{p-2}} \inti{\ud b} \inti{\frac{\ud c}{c} \;c^{-n-\alpha}}
				\lambda(a+b)\left(Dc^{n+\frac{\alpha-1}{2}}-\lambda(b)\right)_+.
\end{align*}
From \cite{LL}, Theorem 1.13, $g^q(x)=\inti{qa^{q-1}1_{\{g(x)>a\}} \ud a}$. Thus, we denote
\[\|g\|_{q(\mu)}^q = \int_{\HN}{|g(x)|^q \dvx} = \inti{qa^{q-1}\lambda(a) \ud a} \ge \lambda(b) b^q.\]
Using the substitution $c = \left(\frac{\lambda(b)}{D}\right)^{\frac{2}{2n+\alpha-1}}t$, and the identity $1-\frac{p}{q} = \frac{\alpha+1}{2n+\alpha-1}$, then
\begin{align*}
J_{\alpha,p}^{\HN}(f)
&\ge		d_{n,p,\alpha} \inti{\ud a \;a^{p-2}} \inti{\ud b \;\lambda(a+b)\lambda(b)^{-\frac{\alpha+1}{2n+\alpha+1}}}
				\int_1^{\infty}{\ud t \;t^{-n-\alpha-1}\left(t^{n+\frac{\alpha-1}{2}}-1\right)}\\
&\ge		d_{n,p,\alpha} \|g\|_{q(\mu)}^{p-q} \inti{\ud a \;a^{p-2}} \int_0^a{\ud b \;\lambda(a+b)b^{q-p}}\\
&\ge		d_{n,p,\alpha} \|g\|_{q(\mu)}^{p-q} \inti{a^{q-1}\lambda(2a) \ud a}\\
&=			d_{n,p,\alpha} \|g\|_{q(\mu)}^{p}\\
&=			d_{n,p,\alpha} \left(\int_{\HN}{|f(x)|^q x_n^{\frac{\alpha-1}{2} - q\frac{\alpha-1}{p}} \ud x}\right)^{p/q}.
\end{align*}
Since $\frac{\alpha-1}{2} - q\frac{\alpha-1}{p} = -n + \frac{nq}{p^*}$, we are done.
\qed

\begin{lm}
\label{lemma}
Let $n \geq 2$, $0<\alpha<2$, and $f \in C_c(\RR)$. Then, $I_{\alpha}^{\RR}(f) = I_{\alpha}^{\RR}(\ft)$, with the understanding that $I_{\alpha}^{\RR}(\ft) = \infty$ if $I_{\alpha}^{\RR}(f) = \infty$.
\end{lm}

\begin{proof}
First, let us define the set
\[A_\epsilon := \left\{(x,y) \in \RR\times\RR : 1-\epsilon < \sqrt{\frac{\eta(y)}{\eta(x)}} < \frac{1}{1-\epsilon}\right\}.\]
Using the transformation $T$, as discussed in the remarks prior to Theorem \ref{rearrange}, and results from \cite{CL2},
\begin{align*}
I_{\alpha}^{\RR}(f)
&= 	\int_{\RR\times\RR}{\frac{\left|f(Tx)-f(Ty)\right|^2}{\left[\eta(x) \left|x-y\right|^2 \eta(y)\right]^{(n+\alpha)/2}}
				\left[\eta(x)\eta(y)\right]^n \ud x \ud y}\\
&=	\int_{\RR\times\RR}{\frac{\left|\eta(x)^{(\alpha-n)/2}\ft(x)-\eta(y)^{(\alpha-n)/2}\ft(y)\right|^2}
				{\left|x-y\right|^{n+\alpha}} \left[\eta(x)\eta(y)\right]^{\frac{n-\alpha}{2}} \ud x \ud y}\\
&=	\lim_{\ez} \int_{(A_\epsilon)^C}{\frac{\ud x \ud y} {\left|x-y\right|^{n+\alpha}}
				\left[\left(\left(\frac{\eta(y)}{\eta(x)}\right)^{\frac{n-\alpha}{2}}-1\right)\ft^2(x)
				+ \left(\left(\frac{\eta(x)}{\eta(y)}\right)^{\frac{n-\alpha}{2}}-1\right)\ft^2(y)+\left(\ft(x)-\ft(y)\right)^2\right]}\\
&=	I_{\alpha}^{\RR}(\ft) + 2\lim_{\ez} \int_{(A_\epsilon)^C}{\frac{\ft^2(x)}{\left|x-y\right|^{n+\alpha}}
				\left[\left(\frac{\eta(y)}{\eta(x)}\right)^{\frac{n-\alpha}{2}}-1\right] \ud x \ud y}.
\end{align*}
We show that the limit on the last line is zero for all $\epsilon > 0$; thus, establishing our result. We write $x=(x',x_n)$, $x'\in\R^{n-1}$, $x_n\in\R$. Then,
\begin{align*}
\int_{(A_\epsilon)^C}{\frac{\ft^2(x)}{\left|x-y\right|^{n+\alpha}}\left[\left(\frac{\eta(y)}{\eta(x)}\right)^{\frac{n-\alpha}{2}}-1\right] \ud x \ud y}
&= 	\int_{(A_\epsilon)^C}{\frac{\ft^2(x',x_n)}{\left|x-y\right|^{n+\alpha}}
				\left[\left(\frac{\left|x'\right|^2+(x_n+1)^2}{\left|y'\right|^2+(y_n+1)^2}\right)^{\frac{n-\alpha}{2}}-1\right] \ud x \ud y}\\
&\mspace{-189.0mu}
		=	\int_{\RR}{\ud x |x|^{n-\alpha} \ft^2(x',x_n - 1)}
				\mspace{-45.0mu} \int_{\left\{y: (1-\epsilon)|x|\leq |y| \leq \frac{|x|}{1-\epsilon}\right\}^C}\mspace{-45.0mu}
				{\ud y \frac{|y|^{\alpha-n}-|x|^{\alpha-n}}{\left|x-y\right|^{n+\alpha}}}\\
&\mspace{-189.0mu}
		=	\left|\SN^{n-2}\right|\left(\int_{\RR}{\ud x \frac{\ft^2(x',x_n - 1)}{|x|^{\alpha}}}\right)
				\left(\int_{\left[1-\epsilon,\frac{1}{1-\epsilon}\right]^C}\mspace{-27.0mu} {\ud t \left(t^{\alpha-1}-t^{n-1}\right)}
				\int_{-1}^1 {\ud s \frac{\left(1-s^2\right)^{(n-3)/2}}{\left(t^2+1-2st\right)^{(n+\alpha)/2}}}\right),
\end{align*}
where the complement in the last integral is with respect to the half line $[0,\infty)$. This product of integrals is zero as the left integral is finite, while the right integral is zero, for all $\epsilon > 0$. Indeed, for the right integral, note that there is no singularity so long as $t \neq 1$.  If we split the integral into $t$ above and below 1, then the latter must be finite, so the integral is finite if the sum is. We compute
\pagebreak[4]
\begin{align*}
\int_{\left[1-\epsilon,\frac{1}{1-\epsilon}\right]^C}\mspace{-27.0mu} {\ud t \left(t^{\alpha-1}-t^{n-1}\right)}
				\int_{-1}^1 {\ud s \frac{\left(1-s^2\right)^{(n-3)/2}}{\left(t^2+1-2st\right)^{(n+\alpha)/2}}}
&=		\int_0^{1-\epsilon}{\ud t \left(t^{\alpha-1}-t^{n-1}\right)}
				\int_{-1}^1 {\ud s \frac{\left(1-s^2\right)^{(n-3)/2}}{\left(t^2+1-2st\right)^{(n+\alpha)/2}}}\\
&			\qquad + \int_{\frac{1}{1-\epsilon}}^{\infty}{\ud t \left(t^{\alpha-1}-t^{n-1}\right)}
				\int_{-1}^1 {\ud s \frac{\left(1-s^2\right)^{(n-3)/2}}{\left(t^2+1-2st\right)^{(n+\alpha)/2}}}.
\end{align*}
In fact, the sum is zero, as
\begin{align*}
\int_{\frac{1}{1-\epsilon}}^{\infty}{\ud t \left(t^{\alpha-1}-t^{n-1}\right)}
				\int_{-1}^1 {\ud s \frac{\left(1-s^2\right)^{(n-3)/2}}{\left(t^2+1-2st\right)^{(n+\alpha)/2}}}
&=		\int_0^{1-\epsilon}{\frac{\ud t}{t^2} \left(t^{1-\alpha}-t^{1-n}\right)}
				\int_{-1}^1 {\ud s \frac{\left(1-s^2\right)^{(n-3)/2}}{\left(1/t^2+1-2s/t\right)^{(n+\alpha)/2}}}\\
&=		\int_0^{1-\epsilon}{\ud t \left(t^{n-1}-t^{\alpha - 1}\right)}
				\int_{-1}^1 {\ud s \frac{\left(1-s^2\right)^{(n-3)/2}}{\left(1+t^2-2st\right)^{(n+\alpha)/2}}}.
\end{align*}
Lastly, we consider the left integral. From \cite{FS1}, there exists $c>0$ so that
\[\int_{\RR}{f^2(x) |x|^{-\alpha} \ud x} \leq c I_{\alpha}^{\RR}(f).\]
Hence, if we assume that $I_{\alpha}^{\RR}(f) < \infty$, then
\[
\int_{\RR}{\ft^2(x',x_n - 1) |x|^{-\alpha} \ud x}
=			2^{-\alpha/2} \int_{\RR}{f^2(Tx) \eta(x)^{n-\alpha/2} \ud x}
=			2^{-\alpha/2} \int_{\RR}{f^2(x',x_n-1) |x|^{-\alpha} \ud x}
\leq		c I_{\alpha}^{\RR}(f) < \infty,
\]
as desired. If $I_{\alpha}^{\RR}(f)$ is not finite, then we need to ask whether $\int_{\RR}{\ft^2(x',x_n - 1) |x|^{-\alpha} \ud x} < \infty$. If so, then the result still holds. If this is not true, then since $\int_{\RR}{\ft^2(x',x_n - 1) |x|^{-\alpha} \ud x} \leq c I_{\alpha}^{\RR}(\ft)$, then $I_{\alpha}^{\RR}(\ft)$ is not finite as well.
\end{proof}

\end{document}